\documentclass[reqno,12pt]{amsart}

\usepackage{color}
\usepackage{amsmath, amsthm, amssymb}
\usepackage{amsfonts}
\usepackage[ansinew]{inputenc}
\usepackage[dvips]{epsfig}
\usepackage{graphicx}
\usepackage[english]{babel}
\usepackage{hyperref}
\theoremstyle{plain}
\newtheorem{thm}{Theorem}[section]
\newtheorem{cor}[thm]{Corollary}
\newtheorem{lem}[thm]{Lemma}
\newtheorem{prop}[thm]{Proposition}

\theoremstyle{definition}
\newtheorem{defi}[thm]{Definition}

\theoremstyle{remark}
\newtheorem{rem}[thm]{Remark}

\numberwithin{equation}{section}

\newcommand{\average}{{\mathchoice {\kern1ex\vcenter{\hrule height.4pt
width 6pt depth0pt} \kern-9.7pt} {\kern1ex\vcenter{\hrule
height.4pt width 4.3pt depth0pt} \kern-7pt} {} {} }}

\def\R{\mathbb{R}}

\begin{document}

\title[Stable $s$-minimal cones in $\mathbb{R}^3$ are flat  for $s\sim 1$]
{Stable $s$-minimal cones in $\mathbb{R}^3$ are flat  for $s\sim 1$}

\author{Xavier Cabr\'e}
\author{Eleonora Cinti}
\author{Joaquim Serra}

\address{X. Cabr\'e \textsuperscript{1,2},
\newline
\textsuperscript{1} Universitat Polit\`ecnica de Catalunya, Departament
de Matem\`{a}tiques, Diagonal 647,
08028 Barcelona, Spain
\newline
\textsuperscript{2} ICREA, Pg. Lluis Companys 23, 08010 Barcelona, Spain}
\email{xavier.cabre@upc.edu}


\address{E. Cinti, Universit\`a degli Studi di Bologna, Dipartimento di Matematica, Piazza di Porta San Donato 5,
40126 Bologna , Italy .}
\email{eleonora.cinti5@unibo.it}
\address{J. Serra, Eidgen\"ossische Technische Hochschule Z\"urich, R\"amistrasse 101, 8092 Zurich, Switzerland}
\email{joaquim.serra@math.ethz.ch}

\thanks{The authors are supported by MINECO grant MTM2014-52402-C3-1-P
and are part of the Catalan research group 2014 SGR 1083. The first author is member of the Barcelona
Graduate School of Mathematics}

\begin{abstract}
We prove that half spaces are the only  stable nonlocal $s$-minimal cones in $\mathbb{R}^3$, for $s\in(0,1)$ sufficiently close to $1$. This is the first classification result of stable $s$-minimal cones in dimension higher than two. Its proof can not rely on a compactness argument perturbing from $s=1$. In fact, our proof gives a quantifiable value for the required closeness of $s$ to $1$. We use the geometric formula for the second variation of the fractional $s$-perimeter, which involves a squared nonlocal second fundamental form, as well as the recent BV estimates for stable nonlocal minimal sets.
\end{abstract}

\maketitle


\section{Introduction}\label{intro}

In this paper we  prove that half spaces are the only stable nonlocal $s$-minimal cones ---with smooth boundary away from $0$---  in dimension $n=3$ for $s\in (0,1)$ sufficiently close to $1$ (see Theorem \ref{thmcones3}). The same classification result for stable $s$-minimal cones in dimension $n=2$ for any $s\in (0,1)$ has been established in \cite{SV-mon}. For short, we will refer to nonlocal $s$-minimal cones as $s$-minimal cones.

For minimizing cones (a stronger assumption than stability) a similar flatness result was proven  by Savin and Valdinoci \cite{SV} in dimension $n=2$ for any $s\in(0,1)$, and  by Caffarelli and Valdinoci \cite{CV} in every dimension $2 \le n \le 7$  for $s\in (0,1)$ sufficiently close to $1$. The result in \cite{CV}  relies on the classification of classical ($s=1$) minimizing cones of Simons \cite{Simons} and extends it to $s$ sufficiently close to $1$ through a compactness argument.

Our statement is also ``for $s$ sufficiently close to $1$'', but ---unlike in \cite{CV}---  it cannot be deduced from the limit case $s=1$  by some sort of compactness argument.
The reason being  that ---unlike in the framework of minimizers---   $E_k$ being stable cones for the $s_k$-perimeter with $s_k\uparrow 1$ does not guarantee the sequence $E_k$ to be  compact. We must rule out, for instance, an hypothetical situation in which the traces of $E_k$ on $S^2$ were (unions of) curves with their total classical perimeter increasing to infinity.
As a matter of fact, and at least in $\R^3$, proving the compactness of sequences $E_k$ of stable cones turns out to be as difficult as proving the flatness result ---which then trivially gives the compactness since planes through the origin are compact.

Let us remark also that the classification of stable cones in low dimensions turns out to be significantly more challenging for $s\in(0,1)$  than in the classical case $s=1$. Indeed, as mentioned above, when $n=2$ the classification of stable $s$-minimal cones ---for all $s\in(0,1)$--- requires already a clever idea \cite{CSV, SV-mon}. Moreover, in  the case $n=3$ of this paper, there is an even larger gap of difficulty between $s\in (0,1)$ and $s=1$. Indeed,  in the classical perimeter case $s=1$,  the trace $\partial \Sigma\cap S^2$ on the sphere  of every stationary\footnote{$\partial \Sigma$ has zero mean curvature} cone $\Sigma\subset \R^3$ with $C^2$ boundary away from $0$ is immediately a maximal circle ---and here the stability assumption is not even required. This is proven just using  that the zero mean curvature condition on $\partial \Sigma$ is equivalent to a zero tangential curvature condition for the $C^2$ curve $\partial \Sigma\cap S^2$.   For $s\in(0,1)$, however, the nonlocal character of the equation of $s$-minimal cones makes it impossible for such sort of ``ODE type'' approach.

Before stating precisely  our main result, we recall the notion of fractional $s$-perimeter, which was introduced by Caffarelli, Roquejoffre, and Savin in \cite{CRS}. Given a set $E$ in $\R^n$ and a bounded open set $\Omega\subset \R^n$, we define the fractional $s$-perimeter of $E$ in $\Omega$ as
\begin{equation}\label{def-per}
P_{s}(E,\Omega):=L(E\cap \Omega,E^c \cap \Omega) + L(E\cap\Omega, E^c \cap \Omega^c) +L(E\cap\Omega^c, E^c \cap \Omega) ,
\end{equation}
where $E^c$ denotes the complement of $E$ in $\R^n$ and, for two disjoint measurable sets $A$ and $B$, $L(A,B)$ denotes the quantity
$$L(A,B):=\int_A\int_B\frac{1}{|x-\bar x|^{n+s}}dx d\bar x.$$

Minimizers for the fractional perimeter, with special interest in their regularity, were studied in several works; see \cite{BFV,CRS,CV,CSV, FV, SV}. However, to our knowledge, the only available results for stable sets of the $s$-perimeter have been obtained recently in \cite{CSV}. This article includes sharp $BV$ and energy estimates in every dimension $n\ge 2$, and quantitative flatness results in dimension $n=2$.

Let us state the main result of the current paper. We say that $\Sigma\subset \R^n$ is a  cone when $\lambda\Sigma =\Sigma$ for all $\lambda>0$.  We will always take $\Sigma$ to be an open set. Its boundary $\partial \Sigma$, a hypersurface in $\R^n$, will also be called a cone.  The following is the definition of stability that we use.

\begin{defi}\label{defstablecone}
Let $\Sigma\subset \R^{n}$ be a cone with nonempty boundary of class $C^2$ away from the origin.
We say that  $\Sigma$  is stable if
\begin{equation}\label{stabilityweak}
\liminf_{t\to 0} \frac{1}{t^2} \big(P_s(\phi^t_X(\Sigma), B_1)-P_s(\Sigma, B_1))\ge  0
\end{equation}
for all vector fields $X\in C^\infty_c(B_1\setminus\{0\}, \R^n)$. Here $\phi^t_X:\R^n\to \R^n$ denotes the integral flow of $X$ at time $t$ (which is a smooth diffeomorphism for $t$ small).

Throughout the paper, $\Sigma$ being stable as in this definition will also be referred to as $\Sigma$ being a stable cone for the $s$-perimeter in $\R^n$, or $\Sigma$ being a stable $s$-minimal cone in $\R^n$.
\end{defi}

Note that $\phi^t_X$  is the identity in a neighborhood of  $0$ and, thus, this is the weakest possible notion of stability of cones that one may assume. In Section 2 we will briefly discuss other notions of stability for the $s$-perimeter.

It is easy to see that if $\Sigma$ is as in Definition \ref{defstablecone} (in particular, $\partial\Sigma$ is $C^2$ away from $0$) then $\partial \Sigma$ is stationary away from $0$, and hence it is a solution of the nonlocal minimal surface equation (also away from $0$). Moreover, using that $\Sigma$ is a cone one can show (see the proof of Theorem \ref{thmcones3} for the details) that $\partial\Sigma$ is a viscosity solution of the nonlocal minimal surface equation also in $0$.

The following is our main result.
\begin{thm}\label{thmcones3}
There exists $s_*\in (0,1)$ such that for every $s\in (s_*,1)$ the following statement holds.

Let $\Sigma \subset \R^3$ be a cone with nonempty boundary of class $C^2$ away from $0$. Assume that $\Sigma$ is stable as in Definition \ref{defstablecone}.
Then, $\Sigma$ is a half space.
\end{thm}

As mentioned before, Theorem \ref{thmcones3} is the first classification result for stable $s$-minimal cones in dimension $n=3$. The analogue result for $n=2$ and for any $s\in(0,1)$ was established in \cite{SV-mon} (see also the quantitative version \cite{CSV}).

We stress that our result is not a perturbative result from $s=1$ which can be obtained by some sort of compactness argument. In fact, a careful inspection of our proof gives an explicit (computable) value for $s_*$, something impossible when using compactness arguments.

A consequence of Theorem \ref{thmcones3} is the following.
\begin{cor}\label{cor}
There exists $s_*\in (0,1)$ such that for every $s\in (s_*,1)$ the following statement holds.

Let $E$ be an open  subset of $\R^3$. Assume that $\partial E$ is nonempty and of class $C^2$, and that $E$ is a stable set for the $s$-perimeter. Then, $E$ is a half space.
\end{cor}


For the reasons explained below, the proof of Corollary \ref{cor} will be given in full detail in the forthcoming paper \cite{CCS}.
It follows a rather standard (at least in the context of minimizers) blow-down approach. 
Besides the classification of stable cones from Theorem \ref{thmcones3}, the proof of Corollary \ref{cor} needs the following four ingredients, which are known in the setting of stable $s$-minimal sets provided that their boundaries are $C^2$:
\begin{itemize}
\item[(i)] Universal perimeter estimate (established for $C^2$ stable sets in \cite{CSV});
\item[(ii)] Density estimates (established for $C^2$ stable sets in \cite{CCS});
\item[(iii)] Monotonicity formula (established for minimizers in  \cite{CRS} with a proof that works also for $C^2$ stable sets);
\item[(iv)] Improvement of flatness (established for minimizers in  \cite{CRS} with a proof that works also for $C^2$ stable ---or even stationary--- sets).
\end{itemize}

Note that in the context of classical minimal surfaces,  (i), (ii), and (iv) are known for minimizers but not for stable sets (at least in dimensions $n>3$). This is why the analogue of Corollary \ref{cor} (that is, the classification of minimal surfaces, and not only cones) in dimensions $3\leq n\leq 7$, is only known for minimizers although stable minimal cones are completely classified.

The main obstruction to remove the $C^2$ assumption is  (iv), since the improvement of flatness in \cite{CRS} has been established for viscosity solutions of the nonlocal minimal surface equation. Although it is obvious that $C^2$ stable $s$-minimal sets are viscosity solutions, the same is not known for generic stable sets.

Properties analogous to (i)-(iv) will appear in our forthcoming paper \cite{CCS} in the context of 
stable solutions to the fractional Allen-Cahn equation
\begin{equation}\label{ACE}
(-\Delta)^{s/2}u=u-u^3,\quad |u|<1\quad \mbox{in }\R^n.
\end{equation} 
We will prove there a classification result analogous to Corollary \ref{cor}, but now for equation \eqref{ACE}.
The proofs for solutions to \eqref{ACE} and for $s$-minimal surfaces 
are roughly the same and, moreover, both use some tools that are not needed in the present paper (e.g., the 
Caffarelli-Silvestre extension). For these reasons, we have decided to differ the details of the proof of 
Corollary \ref{cor} to \cite{CCS}. 

The abstract classification result in \cite{CCS}  (which was actually the primary motivation of the
present paper) is: 

\vspace{6pt}

{\em Assume that for some pair $(n,s)$ the half-spaces are the only   stable $s$-minimal cones in $\R^m$  
(which are smooth away from $0$) for $2\le m \le n$. Then, every stable solution of \eqref{ACE} in $\R^n$ is a 1D profile, 
that is, $u(x)=\phi(e\cdot x)$ for some increasing function $\phi:\R\to (-1,1)$ and $e\in S^{n-1}$.}

\vspace{6pt}

As a consequence of this statement and of Theorem \ref{thmcones3}, 
we establish in \cite{CCS}:
\begin{itemize}
\item [i)] 1D profiles are the only {\em stable} solutions of \eqref{ACE} when $n=3$ and $s\in (0,1)$ is sufficiently
close to $1$;
\item [ii)] 1D profiles are the only  {\em  monotone} solutions of \eqref{ACE} when $n=4$ and $s\in (0,1)$ is 
sufficiently
close to $1$.
\end{itemize}
Previously, 1D symmetry of stable solutions to \eqref{ACE} for $s/2 <1/2$ was only known in dimension 2.

The proof in \cite{CCS} of the classification result for \eqref{ACE} establishes that blow down sequences 
$u(R_k x)$ converge to $\chi_{\Sigma}-\chi_{\Sigma^c}$ where $\Sigma$ is a stable minimal cone which, after a dimension reduction, can be assumed to be smooth away from $0$. Furthermore, in \cite{CCS}  we prove density estimates ensuring the local uniform convergence of the level sets of $u$ to $\partial \Sigma$ (in the sense of the Hausdorff distance). As a consequence, if we know that the cone is a half space, the improvement of flatness theory for ``genuinely nonlocal'' phase transitions established in \cite{dPSV} gives that $u$ must be a 1D profile.

Let us finally comment on the proof of Theorem \ref{thmcones3}. It will use three important ingredients from recent works, namely:
\begin{itemize}

\item[a)] The  second variation formula for the nonlocal perimeter from \cite{DDPW,FFFMM}, which involves a squared nonlocal second fundamental form and that we recall  in Theorem \ref{propFFFMM}.

\item[b)] The behavior as $s\uparrow 1$ of the optimal constant in the fractional Hardy inequality in dimension two, which can be found for instance in \cite{FS}, and which we recall in Theorem \ref{optctthardy}.

\item[c)] The universal $BV$ estimate for stable sets of \cite{CSV}, which we recall in Theorem \ref{perestimate}. In particular, the information that its best constant may be bounded by $\frac{C}{1-s}$ when $s\uparrow 1$.
\end{itemize}

To prove Theorem \ref{thmcones3} we plug in the stability inequality given by a)   a radial function that  ``almost saturates'' the Hardy inequality b) in dimension two. Then, we integrate in the radial variable, and appropriately use the universal BV estimate c) ---at every scale--- to relate the integrals  on $\partial \Sigma$ (a curved two dimensional cone) with the integrals in $\R^2$ appearing in the Hardy inequality. With this, we obtain an integral control on  $\partial \Sigma \cap S^2$ for the nonlocal version of the squared second fundamental form of  $\partial \Sigma$. This control is given in Proposition \ref{crucial3d} and is the main goal of Section \nolinebreak  4.

Concluding the flatness of the cone from the control in Proposition \ref{crucial3d} is not a straightforward task.
To do it, we need a series of lemmas on curves on $S^2$ ---given in Section 5---  the cornerstone of which is Lemma \ref{lemcurv}.  It establishes bounds on the length of a curve on $S^2$ from an integral control on its squared nonlocal second fundamental form.
Interestingly, a crucial ingredient in the proof of  Lemma \ref{lemcurv} is an elementary ``topological'' observation on closed injective curves in the cylinder $S^1\times(-1,1)$, which is given in Lemma \ref{topolog}.
The application of these lemmas gives that, for $s$ close enough to 1, that curve $\partial \Sigma \cap S^2$ is a simple curve that  is very close--- in a $C^{1,1/4}$ norm--- to a maximal circle. We conclude that the curve must be a maximal circle using the classification of $s$-minimal Lipschitz graphs of Savin and Valdinoci \cite{SV}.

The proof of Theorem \ref{thmcones3} is given in Section 6 by combining all the previous results.

\section{On the notion of stable sets for the $s$-perimeter}

Throughout the paper the notion of stability that we consider is the one of Definition \ref{defstablecone}, which is given specifically in the context of cones in $\R^3$ with $C^2$ boundary away from $0$. 
In this setting, Definition \ref{defstablecone}  is the weakest notion of stability one can think of  ---note that we do not need to allow perturbations that affect the vertex of the cone.

For the sake of clarity, we recall now the notion of stability that  was introduced and used  in \cite{CSV}, and we explain below why this was done. It applies to general sets of finite $s$-perimeter. 
\begin{defi}[\cite{CSV}, stability]\label{stability}
A set $E\subset \R^n$  with $P_s(E,\Omega)<\infty$ is said to be {\it stable in} $\Omega$ if for every given vector field $X=X(x,t)\in C^\infty_c(\Omega\times(-1,1); \R^n)$ we have
\[
\liminf_{t\to 0} \frac{1}{t^2} \big( P_s(\phi_X^t(E)\cap E,  \Omega)- P_s(E, \Omega)\big) \ge 0
\]
and
\[
\liminf_{t\to 0} \frac{1}{t^2} \big( P_s(\phi_X^t(E)\cup E,  \Omega)- P_s(E, \Omega)\big) \ge 0,
\]
where $\phi_X^t$ is the integral flow of $X$ at time $t$.
\end{defi}

Another possible notion of stability, which is weaker than the one given in Definition \ref{stability} above, is the following:

\begin{defi}\label{defweak}
A set $E\subset \R^n$  with $P_s(E,\Omega)<\infty$ is said to be {\it weakly stable in} $\Omega$ if for every given vector field $X=X(x,t)\in C^\infty_c(\Omega\times(-1,1); \R^n)$ we have
\[
\liminf_{t\to 0} \frac{1}{t^2} \big( P_s(\phi_X^t(E),  \Omega)- P_s(E, \Omega)\big) \ge 0,
\]
where $\phi_X^t$ is the integral flow of $X$ at time $t$.
\end{defi}

Notice that every stable set $E$ (i.e., satisfying Definition \ref{stability}) is also weakly stable (in the sense of Definition \ref{defweak}),
as it is immediately shown using the inequality
\[
P_s(\phi_X^t(E), \Omega)+ P_s(E, \Omega) \ge P_s(\phi_X^t(E)\cap E,  \Omega) + P_s(\phi_X^t(E)\cup E,  \Omega).
\]


For $s\in(0,1]$, both definitions are known to be equivalent\footnote{See Remark \ref{equivalencedefs}} when applied to sets $E$ with $C^2$ boundary in $ \Omega$. Thus, our stability assumption in Definition \ref{defstablecone} and Theorem \ref{thmcones3} is equivalent to the cone with $C^2$ boundary away from the origin being stable in $\R^n\setminus \{0\}$ in the sense of Definition \ref{stability}, and also to being weakly stable in $\R^n\setminus \{0\}$ as in Definition \ref{defweak}.

Note that, for the classical perimeter ($s=1$), Definition \ref{stability} ---and not Definition \ref{defweak}--- is the correct notion of stability in order to rule out  cones such as ``the cross''
\[
\{ (x_1,x_2)\in \R^2\ :\ x_1x_2>0\}
\]
to be stable. The reason is that Definition \ref{stability} allows ``infinitesimal'' perturbations that ``break the topology'' of $E$ ---while Definition \ref{defweak} does not.

The previous example shows that the two notions of stability are indeed different in the limit case $s=1$ of the classical perimeter.
For $s\in(0,1)$, however,  some heuristics seem to suggest that the two definitions might be equivalent.
For instance, ``the cross'' is no longer weakly stable, due to nonlocal effects.

It is an  open question whether (or not), in the nonlocal case $s\in(0,1)$,  every weakly stable set is stable. This statement, if true, would be very useful to obtain ---using the BV estimates of \cite{CSV}--- clean compactness results for stable sets for the $s$-perimeter ---with $s\in(0,1)$ fixed---, since weak stability is better suited for passages to the limit.

\section{Previously known ingredients that our proof uses}

As explained in the introduction, the proof of Theorem \ref{thmcones3} uses three main ingredients from previous works, which we gather here.

First, we will use a formula, found in   \cite{FFFMM,DDPW}, for  the second (normal) variation of the fractional perimeter. We state it in $\R^3$ but an analogue in $\R^n$ also holds true.
\begin{thm}[\cite{FFFMM, DDPW}] \label{propFFFMM}
Let $\Sigma \subset \R^3$ be a stable cone for the $s$-perimeter.  Assume that $\partial \Sigma$ is $C^2$ away from  $0$. Then, for every $\zeta\in C^2_c(\R^3\setminus \{0\})$ we have
\[
\int_{\partial \Sigma} c^2_{s, \partial \Sigma}(x) |\zeta(x)|^2 \,dH^{2}(x)\leq 
\iint_{\partial \Sigma\times\partial \Sigma}  \frac{	\big|\zeta(x)-\zeta(y)\big|^2 }{|x-y|^{3+s}} \,dH^{2}(x)\,dH^{2}(y),
\]
where
\[
c^2_{s, \partial \Sigma}(x)  :=  \int_{\partial \Sigma} \frac{	\big|\nu_\Sigma(x)-\nu_\Sigma(y)\big|^2 }{|x-y|^{3+s}} \,dH^{2}(y)
\]
and $\nu_\Sigma(x)$ denotes the outward normal vector to $\Sigma$ at $x\in \partial \Sigma$.
\end{thm}

\begin{rem} \label{equivalencedefs}
Theorem \ref{propFFFMM} is an application (to the case of cones in $\R^3$) of a second variation formula found in \cite{FFFMM, DDPW}  for sets $E\subset \R^n$ with $C^2$ boundaries.
Namely, if $X$ is a smooth vector field and $\partial E$ is $C^2$ we have
\begin{eqnarray}
&\hspace{-65mm}\lim_{t\to 0} \frac{1}{t^2} \big( P_s(\phi_X^t(E),  \Omega)- P_s(E, \Omega)\big)  =\label{stability11}\\
&\hspace{3mm}=
\displaystyle \iint_{\partial E\times\partial E}  \frac{	\big|\zeta(x)-\zeta(y)\big|^2 }{|x-y|^{n+s}} \,dH^{n-1}(x)\,dH^{n-1}(y) - \int_{\partial \Sigma} c^2_{s, \partial \Sigma}(x) |\zeta(x)|^2 \,dH^{n-1}(x),\nonumber
\end{eqnarray}
where $\zeta=X\cdot \nu_E$. A standard approximation argument then shows that  \eqref{stability11} holds for all $\zeta$ Lipschitz and compactly supported on $\partial E$.

Using this formula we can show that, in the class of $C^2$ sets, the  two notions  of stability in Definitions \ref{stability} and \ref{defweak} are equivalent. Indeed, when $\partial E$ is $C^2$ we have 
\[
\begin{split}
&\liminf_{t\to 0} \frac{1}{t^2} \big( P_s(\phi_X^t(E)\cap E,  \Omega)- P_s(E, \Omega)\big)  =\\
&\hspace{3mm}= 
\iint_{\partial \Sigma\times\partial \Sigma}  \frac{	\big|\zeta^-(x)-\zeta^-(y)\big|^2 }{|x-y|^{n+s}} \,dH^{n-1}(x)\,dH^{n-1}(y) - \int_{\partial \Sigma} c^2_{s, \partial \Sigma}(x) |\zeta^-(x)|^2 \,dH^{n-1}(x),
\end{split}
\]
where $\zeta^-$ denotes the negative part of  $\zeta=X\cdot \nu_E$. The same holds with $\cap$ replaced by $\cup$ and the negative part replaced by the positive part. From these observations, it follows that  the stronger notion of stability (Definition \ref{stability}) holds whenever the weaker definition of stability (Definition \ref{defweak}) holds. 
\end{rem}

We recall now the precise dependence on the power $\sigma$ as $\sigma \uparrow 1$ in the definition of the fractional Laplacian in $\R^d$:
$$(-\Delta)^\sigma \zeta(x) =c_{d,\sigma}\int_{\R^d}\frac{\zeta(x)-\zeta(y)}{|x-y|^{d+2\sigma}}\,dy,$$
where   
\begin{equation}\label{1-s}c_{d,\sigma}=2^{2\sigma} \pi^{-d/2}\frac{\Gamma(d/2+\sigma)}{-\Gamma(-\sigma)}=2^{2\sigma} \pi^{-d/2}\frac{\Gamma({d/2+\sigma)}}{\Gamma(2-\sigma)}\sigma (1-\sigma).\end{equation}
In particular,  we observe that, up to a positive multiplicative constant, $c_{d,\sigma}$ behaves like $1-\sigma$ as $\sigma \uparrow 1$. Note also that integration by parts yields

%
\begin{equation}\label{fractLapl}\begin{split}
\int_{\R^d} \zeta(x)(-\Delta)^\sigma \zeta(x)\,dx &=c_{d,\sigma}\int_{\R^d}\zeta(x)\left(\int_{\R^d}\frac{\zeta(x)-\zeta(y)}{|x-y|^{d+2\sigma}}\,dy\right)\,dx\\
&=\frac{c_{d,\sigma}}{2}\int_{\R^d}\int_{\R^d}\frac{|\zeta(x)-\zeta(y)|^2}{|x-y|^{d+2\sigma}}\,dx\,dy.
\end{split}\end{equation}

Our proof requires the knowledge of the behaviour as $\sigma\uparrow 1$ of the best constant in the Hardy-Rellich inequality involving the $H^{\sigma}$ seminorm---see\footnote{
Note that there is a typo in the expression for the optimal constant in formula (1.6) in \cite{FS}. 
} for instance \cite{FS}. We will also use the fact that the inequality is (almost) saturated by radial $C^\infty_c(\R^d\setminus\{0\})$ functions. 
That radial  functions saturate the inequality  is proved in Section 3.3 of \cite{FS}.  Moreover,  by a standard approximation argument, we can choose these radial functions to be smooth and identically zero in a neighborhood of the origin, since points have zero $H^{\sigma}$ capacity in $\R^d$ for $d\ge 2$.
\begin{thm}[see \cite{FS}]\label{optctthardy}
Given $d\ge 2$ and $0<\sigma<d/2$, the inequality
\[
   \mathcal H_{d,\sigma} \int_{\R^d} \frac{|u(x)|^2} {|x|^{2\sigma}}\,dx \leq \int_{\R^d} u(x) (-\Delta)^{\sigma} u(x)\,dx  
\]
holds for every $u\in H^{\sigma}(\R^d)$ with optimal constant
\begin{equation}\label{costantH}
\mathcal H_{d,\sigma} = 2^{2\sigma} \frac{\Gamma^2\big( d/4+\sigma/2\big)}{\Gamma^2\big( d/4-\sigma/2\big)} =  2^{2\sigma-2}(d/2-\sigma)^2 \frac{\Gamma^2\big( d/4+\sigma/2\big)}{\Gamma^2\big( d/4-\sigma/2 +1\big)},
\end{equation}
where $\Gamma$ is the Gamma function.

Moreover, for every $\epsilon >0$ there exists a nontrivial (not identically zero) radial function $\zeta = \zeta(|x|) \in C^\infty_c(\R^d\setminus\{0\})$ such that
\[
\int_{\R^d} \zeta(x) (-\Delta)^{\sigma} \zeta(x)\,dx  \le (\mathcal H_{d,\sigma} +\epsilon) \int_{\R^d} \frac{|\zeta(x)|^2} {|x|^{2\sigma}}\,dx.
\]
\end{thm}

Next, we rewrite in polar coordinates the last inequality of the theorem, for $d=2$. 

 
\begin{cor} \label{corhardy}
Let $\sigma \in (1/2, 1)$. There exists a radial function $\zeta \in  C^\infty_c\big( (0,+\infty)\big)$, $\zeta \not\equiv 0$, such that
\[
I[\zeta,\sigma] \le C(1-\sigma) J[\zeta,\sigma]
\]
where
\begin{equation}\label{Izeta}
I[\zeta, \sigma] := \int_0^\infty  dr\, r^{1-2\sigma} \int_{1}^\infty  d\tau \, \tau  \big|\zeta(r)-\zeta(r\tau)\big|^2   \iint_{S^1\times S^1} \frac{dH^1(\hat X)   dH^1(\hat Y) }{|\hat X-\tau \hat Y|^{2+2\sigma}},
\end{equation}
\begin{equation}\label{Jzeta}
J[\zeta, \sigma] :=  \int_0^\infty dr\,  r^{1-2\sigma}  |\zeta(r)|^2,
\end{equation}
and  $C$ is a universal constant (in particular, independent of $\sigma$).
\end{cor}
\begin{proof}
First, we observe that the best constant in Theorem \ref{optctthardy} for $d=2$ satisfies $\mathcal H_{2,\sigma} \le C(1-\sigma)^2$ (where $C$ is a positive universal constant) as one can see from the last expression in \eqref{costantH}.

Combining equality \eqref{fractLapl}, where $c_{d,\sigma}$ is given by \eqref{1-s},
and the second inequality of Theorem \ref{optctthardy} with the choice $\epsilon = C(1-\sigma)^2$, we deduce that there is a radial function $\zeta= \zeta(|x|)\in C^\infty_c(\R^d\setminus\{0\})$, with $\zeta \not\equiv 0$, satisfying
\begin{equation}\label{eq*}
 C (1-\sigma)^2 \int_{\R^2} \frac{|\zeta(|x|)|^2} {|x|^{2\sigma}} \,dx \ge (1-\sigma) \iint_{\R^2\times \R^2} \frac{\big|\zeta(|x|)-\zeta(|y|)\big|^2}{|x-y|^{2+2\sigma}}\,dx\,dy.
\end{equation}

Finally, we use polar coordinates  $x = r\hat X$, $y= t\hat Y$ where $r,t\in (0,+\infty)$ and $\hat X, \hat Y \in S^1$, and we integrate only in the set $\{|x|\leq |y|\}$ on the right hand side of \eqref{eq*}, to get
\begin{equation*} \begin{split}
&C(1-\sigma) \int_0^\infty dr\, r^{1-2\sigma}   |\zeta(r)|^2 \\
&\hspace{2em}\ge  \int_0^\infty dr\,  r  \int_{r}^\infty  dt\, t \int_{S^1} dH^1(\hat X) \int_{S^1} dH^1(\hat Y) \frac{ \big|\zeta(r)-\zeta(t)\big|^2 }{r^{2+2\sigma}\left|\hat X-\frac{t}{r} \hat Y\right|^{2+2\sigma}}\\
& \hspace{2em}=
\int_0^\infty dr\,  r^{1-2\sigma}  \int_{1}^\infty  d\tau \,\tau \int_{S^1} dH^1(\hat X) \int_{S^1} dH^1(\hat Y) \frac{ \big|\zeta(r)-\zeta(r\tau)\big|^2 }{|\hat X-\tau \hat Y|^{2+2\sigma}},
\end{split}
\end{equation*}
where in the last equality we have used the change of variables $t=r\tau$.
This concludes the proof of the corollary.
\end{proof}

Finally, we state the estimate established in \cite{CSV} for the classical perimeter of stable sets for the $s$-perimeter, keeping track on how the constant in the estimate blows up as $s\uparrow1$. 

\begin{thm}[\cite{CSV}] \label{perestimate}
Let $E\subset \R^n$ be a stable set for the $s$-perimeter in $B_r(z)$, where $z\in \R^n$, $r>0$. Assume that $\partial E$ is $C^2$ in that ball.

Then,
\[{\rm Per}_{B_{r/2}(z)}(E)  \le \frac{C}{1-s} \,r^{n-1},\]
where ${\rm Per}_{B_{r/2}}(E)$ denotes the relative (classical) perimeter of $E$ in $B_{r/2}(z)$ and $C=C(n,s)$ is bounded as $s\uparrow 1$.
\end{thm}
\begin{proof}
The theorem follows by inspection of the proof of Theorem 1.7 in \cite{CSV}, taking into account the explicit dependence of the constants on $s$ as $s\uparrow1$.
For the sake of clarity, we write here below the crucial estimates in the proof of Theorem 1.7 in \cite{CSV}, with the precise dependence of all constants on $s$, as $s\uparrow 1$. In the sequel $C$ will denote positive constants depending only on $n$ and $s$ (possibly different ones) which remain bounded as $s\uparrow 1$.

In \cite{CSV}, Theorem 1.9 (applied to the kernel $K(z)=|z|^{-n-s}$) gives that
\begin{equation}\label{key}
{\rm Per}_{B_{1}}(E) \leq C\left(1+\sqrt{P_s(E,B_4)}\right)
\end{equation}
if $E$ is a stable set in $B_4$,
where $C$  only depends on $n$ (see  Theorem 1.9 in \cite{CSV} and the value of the constant $C_1$ given at page 6 in \cite{CSV}).
We rewrite now inequalities (3.8) and (3.9) of \cite{CSV} keeping track of the dependence on $s$ of all constants. We have
\begin{equation}\label{key1}
\begin{split}
(1-s)P_s(E,B_4)& \leq (1-s)\iint_{B_4\times B_4}\frac{|\chi_E(x)-\chi_E(y)|}{|x-y|^{n+s}}\,dx\,dy \\
&\hspace{1em} + 2(1-s)\iint_{B_4\times B_4^c}\frac{1}{|x-y|^{n+s}}\,dx\,dy\\
&\leq (1-s)\iint_{B_4\times B_4}\frac{|\chi_E(x)-\chi_E(y)|}{|x-y|^{n+s}}\,dx\,dy + C\\
&\leq C\left( {\rm Per}_{B_4}(E) +1\right),
\end{split}
\end{equation}
with $C=C(n,s)$ bounded as $s\uparrow 1$, where for the last inequality we refer to Theorem \nolinebreak  1 and Remark 5 in \cite{BBM} or to the proof of Proposition 2.2 in \cite{Hitch}.

Hence, \eqref{key}, \eqref{key1}, and Young inequality lead to
\begin{equation}\label{key3}
\begin{split}
{\rm Per}_{B_1}(E)&\leq C\left( 1+ \frac{1}{(1-s)^{1/2}} \big( 1+ {\rm Per}_{B_4}(E)\big)^{1/2}\right)\\
&\leq C\left(1+\frac{1}{\delta(1-s)} +\delta \right) + \delta {\rm Per}_{B_4}(E),
\end{split}
\end{equation}
for all $\delta>0$.

Arguing exactly as in the end of the proof of Theorem 1.7 in \cite{CSV} (that is, rescaling and using the abstract Lemma 3.1 in \cite{CSV}), we deduce that
$${\rm Per}_{B_1}(E)\leq \frac{C}{1-s}.$$
Thus, after rescaling, we conclude the statement of the theorem.
\end{proof}

\section{Bounding the squared nonlocal second fundamental form}

In this section we denote by $\gamma$ the intersection of the boundary of the cone $\partial \Sigma \subset \R^3$ and the sphere $S^2=\{x\in \R^3\, : \, |x|=1\}$. Note that $\gamma$ is a finite union of $C^2$ simple curves on $S^2$.

In the following lemma we compute the stability formula of Theorem \ref{propFFFMM} for a radial test function $\zeta=\zeta(r)$, where $r= |x|$.
\begin{lem}\label{step1}
Let $\Sigma$ be a stable cone for the $s$-perimeter in $\R^3$. Assume that $\partial \Sigma$ is $C^2$ away from $0$. Let us call $\gamma := \partial{\Sigma}\cap S^2$, where $S^2:= \{x\in \R^3\, :\, |x|=1\}$.
Then, for every $\zeta \in C^2_c((0,+\infty))$ we have
\[
 A \, J[\zeta,{\textstyle \frac{1+s}{2}}] \le  \int_0^\infty dr\,r^{-s} \int_{1}^\infty   d\tau\, \tau  \big|\zeta(r)-\zeta(r\tau)\big|^2 \iint_{\gamma\times \gamma}  \frac{ dH^1(\hat x) dH^1(\hat y)}{|\hat x-\tau \hat y|^{3+s}},
\]
where
\[
A :=  \int_{\gamma} dH^1(\hat x)  \,c^2_{s, \partial \Sigma}(\hat x)
\]
and $J[\zeta, \frac{1+s}{2}]$ is given by \eqref{Jzeta}.
\end{lem}
\begin{proof}
We take the radial test function $\zeta = \zeta(|x|)$ in the stability inequality of Theorem \ref{propFFFMM},
and we use polar coordinates $x= r\hat x$,  $y = t\hat y$ to obtain
\[\begin{split}
\int_{\gamma} dH^1(\hat x) \int_{\gamma} dH^1(\hat y)\int_0^\infty dr\,r \int_{0}^\infty dt\,t & \frac{	\big|\zeta(r)-\zeta(t)\big|^2 }{|r\hat x-t \hat y|^{3+s}}  \ge
\\
&\ge  \int_{\gamma}dH^1(\hat x) \int_0^\infty  dr\,r \,c^2_{s, \partial \Sigma}(r\hat x) |\zeta(r)|^2.
\end{split}
\]

We observe that since $\Sigma$ is a cone, $\nu_\Sigma (r\hat x) = \nu_\Sigma (\hat x)$ for all $r>0$ and thus, denoting $\hat y=y/|y|$,
\[
\begin{split}
c^2_{s, \partial \Sigma}(r\hat x) &=  \int_{\partial \Sigma} \frac{	\big|\nu_\Sigma(\hat x)-\nu_\Sigma(\hat y)\big|^2 }{|r\hat x-y|^{3+s}} \,dH^{2}(y) =
\int_{\partial \Sigma} \frac{	\big|\nu_\Sigma(\hat x)-\nu_\Sigma(\hat z)\big|^2 }{|r\hat x-r\hat z|^{3+s}} \,r^2 \,dH^{2}(\hat z)
\\
&=  \frac{c^2_{s, \partial \Sigma}(\hat x)}{r^{1+s}}.
\end{split}
\]

Hence, using that $\partial \Sigma\times \partial \Sigma = \{|x|>|y|\}\cup \{|y|>|x|\}$ up to measure zero sets, and the symmetry of the integrand with respect to  interchanging $x,y$ we obtain
\[\begin{split}
2\int_{\gamma} dH^1(\hat x) \int_{\gamma} dH^1(\hat y)\int_0^\infty dr\,r \int_{r}^\infty dt\,t \frac{	 \big|\zeta(r)-\zeta(t)\big|^2 }{|r\hat x-t \hat y|^{3+s}}  \ge
 A \int_0^\infty \frac{dr}{r^{s}}  |\zeta(r)|^2,
\end{split}
\]
where $A =  \int_{\gamma} dH^1(\hat x)  \,c^2_{s, \partial \Sigma}(\hat x)$.

Doing the change of variables $t = r\tau$, we obtain
\[\begin{split}
& 2\int_{\gamma} dH^1(\hat x) \int_{\gamma} dH^1(\hat y)\int_0^\infty dr\,r \int_{1}^\infty d\tau \frac{r^2}{r^{3+s}} \tau  \frac{	\big|\zeta(r)-\zeta(r\tau)\big|^2 }{|\hat x-\tau \hat y|^{3+s}}  \\
&\hspace{4em}\ge
 A \int_0^\infty \frac{dr}{r^s} |\zeta(r)|^2,
\end{split}
\]
and the lemma follows recalling the definition of $J[\zeta, \frac{1+s}{2}]$ in \eqref{Jzeta}.
\end{proof}

The following lemma allows to estimate the integral on $\gamma\times \gamma$, appearing in the previous lemma, by the integral on $S^1\times S^1$ that appears in $I[\zeta, \sigma]$  of Corollary \ref{corhardy} for $\sigma= \frac{1+s}{2}$. Here it is crucial to use in every ball $B_r(\hat x)$, with $\hat x\in \gamma$ and $r\in(0,1/2)$ the universal perimeter estimate from Theorem \ref{perestimate}.
\begin{lem}\label{lemcompareS1}
Let $\Sigma$ be a stable cone for the $s$-perimeter in $\R^3$, of class $C^2$ away from $0$. Let us call $\gamma := \partial{\Sigma}\cap S^2$, where $S^2:= \{x\in \R^3\, :\, |x|=1\}$.

Then, for all   $\tau >1$ we have
\[
\iint_{\gamma\times\gamma}  \frac{dH^1(\hat x)  dH^1(\hat y)}{|\hat x -\tau\hat y|^{3+s} } \le \frac{C H^1(\gamma)}{1-s}
\int_{S^1\times S^1}  \frac{dH^1(\hat X)  dH^1(\hat Y)}{|\hat X - \tau\hat Y|^{3+s} },
\]
where $S^1:= \{X\in \R^2\, :\, |X|=1\}$ and $C= C(s)$ is bounded as $s\uparrow 1$.
\end{lem}
\begin{proof}
Applying Theorem \ref{perestimate} to the stable cone $\Sigma$, we obtain that,  for all $\hat x \in \gamma$ and $r\in(0,1/2)$, we have
\begin{equation}\label{estimateper}
 H^1\big( \gamma\cap  B_{r}(\hat x)\big)  \le \frac{C}{1-s} \,r,
 \end{equation}
where $C$ denotes a constant depending only on $s$ which is bounded as $s\uparrow 1$.
In particular, by a covering argument we obtain $H^1( \gamma) \le C/(1-s)$.

We now take any couple of points $\hat x\in \gamma$ and $\hat X\in S^1:= \{X\in \R^2\, :\, |X|=1\}$. Let us show that, for all $\tau>1 $,
\begin{equation}\label{blabla1}
 \int_{\gamma} dH^1(\hat y)\,  \frac{1}{|\hat x -\tau\hat y|^{3+s} } \le \frac{C}{1-s}     \int_{S^1} dH^1(\hat Y)\,  \frac{1}{|\hat X -\tau\hat Y|^{3+s} }.
\end{equation}
Indeed, we use a dyadic ring decomposition
\[
\gamma\setminus\{\hat x\} = \bigcup_{-\infty\le k \le {1}} \mathcal A_k \quad \mbox{where }\mathcal A_k = \gamma \cap \big( B_{2^k}(\hat x)\setminus B_{2^{k-1}}(\hat x)\big).
\]
Using \eqref{estimateper} we obtain
\[
 H^1\big( \mathcal A_k\big) \le  \frac{C}{1-s}\, 2^k.
\]
Then, using that
\[
|\hat x -\tau \hat y|^2  = 1+ \tau^2 -2\tau\hat x \cdot \hat y = (\tau-1)^2 + 2\tau (1- \hat x \cdot \hat y)
\]
and that $2^{-3}2^{2k}\leq 2^{-1}|\hat x-\hat y|^2=1-\hat x\cdot \hat y\leq 2^{-1}2^{2k}$ for $y \in \mathcal A_k$, we obtain
\[
\begin{split}
 \int_{\gamma} dH^1(\hat y)\,  \frac{1}{|\hat x -\tau\hat y|^{3+s} } & = \sum_{-\infty\le k\le {1}} \int_{\mathcal A_k} dH^1(\hat y)\,  \frac{1}{|\hat x -\tau\hat y|^{3+s} }
\\
&\le \sum_{-\infty\le k\le {1}} H^1(\mathcal A_k)  \frac{C}{\big( (\tau-1)^2 + \tau 2^{2k} \big)^{\frac{3+s}{2}} }
\\
& \le \sum_{-\infty\le k\le {1}}   \frac{C}{1-s}\, 2^k \frac{1}{\big( (\tau-1)^2 + \tau 2^{2k} \big)^{\frac{3+s}{2}} }
\\
&\le  \frac{C}{1-s}  \int_{S^1} dH^1(\hat Y)\,  \frac{1}{|\hat X -\tau\hat Y|^{3+s} },
\end{split}
\]
where the last inequality follows from the previous considerations applied with $(\Sigma, \gamma)$ replaced by $(\R^3_+, S^1)$.

The lemma then follows integrating  \eqref{blabla1} with respect to $\hat x$ and $\hat X$.
\end{proof}

We can now give the proof of the key integral estimate on $\gamma$ of the squared nonlocal second fundamental form of $\partial \Sigma$.

Let us compute $c^2_{s, \partial \Sigma}(\hat x)$ in terms of only the trace $\gamma = \partial\Sigma\cap S^2$. Recall that $c^2_{s,\partial \Sigma}$ was defined in Theorem \ref{propFFFMM}.
For this, we introduce the kernel
\begin{equation} \label{kernelk}
k_s(\hat x,\hat y) := \int_0^\infty \frac{tdt }{|\hat x-t\hat y|^{3+s}}   = \int_0^\infty \frac{tdt }{\big( t^2+1-2t \hat x\cdot\hat y \big)^{\frac{3+s}{2}} },
\end{equation}
and we note that, since $\Sigma$ is a cone,
\[
\begin{split}
c^2_{s, \partial \Sigma}(\hat x)  &=  \int_{\partial \Sigma} \frac{	\big|\nu_\Sigma(\hat x)-\nu_\Sigma(y)\big|^2 }{|\hat x-y|^{3+s}} \,dH^{2}(y)
\\
&= \int_{\gamma}  \,dH^{1}(\hat y)  \int_0^\infty dt \,t \,  \frac{	\big|\nu_\Sigma(\hat x)-\nu_\Sigma (\hat y)\big|^2 }{|\hat x-t\hat y|^{3+s}}
\\
&= \int_{\gamma}   \big|\nu_\Sigma (\hat x)-\nu_\Sigma (\hat y)\big|^2 k_s(\hat x, \hat y) \,dH^{1}(\hat y),
\end{split}
\]
where $\nu_\Sigma$ is the exterior normal vector to $\partial \Sigma$.

We can now state the key integral estimate from which we will deduce our main theorem.

\begin{prop}\label{crucial3d}
Let $\Sigma$ be a stable cone for the $s$-perimeter in $\R^3$, and of class $C^2$ away from $0$. Let us call $\gamma := \partial{\Sigma}\cap S^2$, where $S^2:= \{x\in \R^3\, :\, |x|=1\}$.
Then,
\[
\int_{\gamma} c^2_{s,\partial \Sigma}(\hat x) \, dH^1(\hat x)   \le CH^1(\gamma),
\]
that is
\begin{equation}\label{main-est}
\iint_{\gamma\times\gamma}|\nu_\Sigma(\hat x)-\nu_\Sigma(\hat y)|^2 k_s(\hat x,\hat y)d H^1(\hat x)\,dH^1(\hat y)\le CH^1(\gamma),
\end{equation}
where $C=C(s)$ is bounded as $s\uparrow 1$.
\end{prop}
\begin{proof}
Let $\zeta = \zeta(|x|)$ be a radial $C^2_c((0,+\infty))$ test function.  Using Lemma \ref{step1}, we obtain
\[
 A \, J[\zeta,{\textstyle \frac{1+s}{2}}] \le  \int_0^\infty dr\,r^{-s} \int_{1}^\infty   d\tau\,\tau  \big|\zeta(r)-\zeta(r\tau)\big|^2\iint_{\gamma\times\gamma}  \frac{dH^1(\hat x)  dH^1(\hat y)}{|\hat x -\tau\hat y|^{3+s} },
\]
where $A =  \int_{\gamma} dH^1(\hat x)  \,c^2_{s, \partial \Sigma}(\hat x)$ and $J[\zeta, \frac{1+s}{2}]$ is given by \eqref{Jzeta}.

Next, applying  Lemma \ref{lemcompareS1}, we deduce that
\[
\begin{split}
 \int_0^\infty dr\,r^{-s}  &\int_1^\infty   d\tau\,\tau  \big|\zeta(r)-\zeta(r\tau)\big|^2 \iint_{\gamma\times \gamma}\frac{  dH^1(\hat x) dH^1(\hat y) }{|\hat x-\tau \hat y|^{3+s}}
\\
&\le \frac{C H^1(\gamma)}{1-s}
\int_0^\infty  dr\,r^{-s} \int_1^\infty   d\tau\,\tau  \big|\zeta(r)-\zeta(r\tau)\big|^2
\iint_{S^1\times S^1} \frac{dH^1(\hat X)   dH^1(\hat Y) }{|\hat X-\tau \hat Y|^{3+s}}
\\
&= \frac{C H^1(\gamma)}{1-s}\, I[\zeta, {\textstyle \frac{1+s}{2}}],
\end{split}
\]
where $I[\zeta,\sigma]$ is as in \eqref{Izeta}.

Therefore, we have
\[
 A \, J[\zeta,{\textstyle \frac{1+s}{2}}] \le  \frac{C H^1(\gamma)}{1-s}\, I[\zeta, {\textstyle \frac{1+s}{2}}].
\]
Finally, choosing $\zeta\not\equiv0$ as in Corollary \ref{corhardy} (with $\sigma = \frac{1+s}{2}$) we have that $I[\zeta, {\textstyle \frac{1+s}{2}}] \le C(1-s) J[\zeta, {\textstyle \frac{1+s}{2}}]$. Since $\zeta\not\equiv0$, the proposition follows combining the last two inequalities.
\end{proof}

The following lemma gives a lower bound for $k_s$.
\begin{lem}\label{lemks}
For $s\in(1/2,1)$, we have
\[
k_s(\hat x,\hat y)  \ge c  \frac {1}{|\hat x-\hat y|^{2+s}} \quad \mbox{for all }\hat x,\,\hat y\in S^2
\]
and for some universal constant $c>0$.
\end{lem}
\begin{proof}
Let us call $b^2 := 1- \hat x\cdot\hat y =\frac{1}{2} |\hat x-\hat y|^2$. Note that $b\in(0,\sqrt 2)$. We have
\[
\begin{split}
k_s(\hat x,\hat y) & = \int_0^\infty \frac{tdt }{\big( (t-1)^2+ 2t(1- \hat x\cdot\hat y) \big)^{\frac{3+s}{2}} }
\ge  \int_{\frac 1 2}^{\frac 3 2} \frac{(1/2)dt }{\big( (t-1)^2+ 3b^2 \big)^{\frac{3+s}{2}} }
\\
&\geq \frac{1}{2} \int_{-\frac{1}{2b}}^{\frac{1}{2b}} \frac{b d\bar t }{\big( (b\bar t)^2+ 3b^2 \big)^{\frac{3+s}{2}} }
\ge \frac{1}{2 b^{2+s}} \int_{-1/4}^{1/4} \frac{d\bar t }{\big( {\bar t}^2+ 3\big)^{\frac{3+s}{2}} }
\\
& \geq   \frac {c}{|\hat x-\hat y|^{2+s}},
\end{split}
\]
where, in the second inequality, we have used the change of variables $\bar t=(t-1)/b$. This concludes the proof of the lemma.
\end{proof}

We observe that, if a connected component $\gamma_0$ of $\gamma$ is parametrized by arc length, then
\begin{equation}\label{ineqarcl}
k_s(\gamma_0(t),\gamma_0(\bar t))\geq \frac{c}{|\gamma_0(t)-\gamma_0(\bar t)|^{2+s}}\geq \frac{c}{|t-\bar t|^{2+s}},
\end{equation}
where we have used Lemma \ref{lemks} for the first inequality and that $|\gamma_0(t)-\gamma_0(\bar t)|\leq |t-\bar t|$ for the second inequality.

We conclude this section with the following embedding.

\begin{lem}\label{poincare}
Let $\sigma\in[3/4,1)$ and $I=[0,5\pi]$. Given  $f: I \rightarrow \R$  we have
\[
\| f -\overline f\|_{C^{1/4}(I)} \leq C \big[f\big]_{H^{\sigma}(I)} 
\]
where $\overline f = \frac{1}{5\pi}\int_I f$,
\[
\big[f\big]_{H^{\sigma}(I)}  := \left( (1-\sigma) \int_I \int_I \frac{|f(t)-f(\bar t)|^2}{|t-\bar t|^{1+2\sigma}}\,dt\, d\bar t \right) ^{1/2},
 \]
and $C$ is a universal constant.
\end{lem}
\begin{proof}
Let  us denote $\big\|f\big\|_{H^{\sigma}(I)} = \|f\|_{L^2(I)} + \big[f\big]_{H^{\sigma}(I)}$. Since $\sigma\ge 3/4$ we have that $H^{\sigma}(I)$ is is continuously embedded in $C^{1/4}(I)$. Then, using the fractional Poincar\'e inequality (see e.g. the ``Fact'' stated in page 80 of \cite{BBM2}) in the interval $I$ we obtain
\[
\begin{split}
\| f -\overline f\|_{C^{1/4}(I)} &\le C \| f -\overline f\|_{H^{3/4}(I)}  = C\{ \| f -\overline f\|_{L^2(I)} +  \big[f\big]_{H^{3/4}(I)}\} \\ &\le C  \big[f\big]_{H^{3/4}(I)}\leq C  \big[f\big]_{H^{\sigma}(I)},
\end{split}
\]
with $C$ universal. We have used, in the last inequality, Remark 5 in \cite{BBM}.
\end{proof}

\section{Auxiliary results on curves of $S^2$}
In this section we prove geometric estimates for a simple curve $\gamma_0$ in $S^2$ satisfying the curvature bounds from Proposition \ref{crucial3d}.

Recall that, throughout the paper,  the trace on $S^2$ of $\partial\Sigma$, which we call $\gamma$, is  (since $\Sigma$ is $C^2$ away from $0$) a finite union of $C^2$ simple closed curves on $S^2$. Moreover, by the perimeter estimate of Theorem \ref{perestimate} we know that the total length of $\gamma$ is bounded by $C(1-s)^{-1}$.
In addition, we obtained in Proposition \ref{crucial3d} a certain integral control on the squared nonlocal second fundamental form of $\partial \Sigma$.

Lemmas \ref{lemcurv} and \ref{lemcurv2} below contain  geometric estimates  for a closed simple curve (i.e., without self-intersections) $\gamma_0$ in $S^2$,  whose length is bounded by $C(1-s)^{-1}$ and satisfying an  integral control on its squared nonlocal second fundamental form. A crucial point is that the constants  in these estimates do not blow up as $s\uparrow 1$.  In the proof of Theorem \ref{thmcones3} these lemmas will be applied to the connected components of $\gamma$.

The first and most important estimate is the following bound, uniform as $s\uparrow1$, for the length of $\gamma_0$.

\begin{lem}\label{lemcurv}
Let $s\in(1/2,1)$,  $L>0$, $\gamma_0= \gamma_0(t) : [0,L] \rightarrow S^2$  be some $C^2$ closed curve without self-intersections and  parametrized by arc length ---thus $L = {\rm length}(\gamma_0)$.
Let $\nu = \gamma_0 \wedge \gamma_0'$  be the ``clockwise'' normal vector (which is tangent to the sphere).

Assume that, for some positive constant $C_0$,
\begin{equation}\label{hp-lemcurv}
\int_0^L \int_0^L \big|\nu(t)-\nu(\bar t)\big|^2 k_s\big(\gamma_0(t), \gamma_0(\bar t)\big)\, dt\, d\bar t\le C_0 L <+\infty.
\end{equation}
 Assume in addition that
\[ 0< L\le \frac{C_0}{1-s}.\]

Then,
\[  L\le C,\]
for some constant $C$ depending only on $C_0$.
\end{lem}

To prove Lemma \ref{lemcurv}, the following ``topological observation will be crucial.
\begin{lem}\label{topolog}
Let   $S^1\times (-1,1)$ be  the  cylinder
\[  \{(x,y,z)\in \R^3 \ :\  x^2 +y^2 =1\ , \ |z|<1\}.\]
Let $\theta\in \R$ (mod $2\pi\mathbb Z$) and $z\in (-1,1)$ be the standard cylindrical coordinates.

Assume that $ \omega : [0,4\pi] \rightarrow S^1\times (-1,1)$  is a $C^1$ curve of the type
\[
\omega=\omega(\theta)= \big(\cos {\theta}, \sin {\theta},  z({\theta})\big)
\]
and satisfying,  for some $b\in \big(0, \frac{1}{100}\big)$,
\[
|z(0)| \le \frac{b}{2}   \qquad \mbox{and} \qquad |z'(\theta)|\le \frac{b}{8\pi} \quad \mbox{for all }\theta\in[0,4\pi].
\]
Assume in addition that $\omega$ is injective ---i.e., it does not have self intersections--- and that $z(0)<z(4\pi)$.

Let  $\tilde \omega : [t_1, t_2] \rightarrow S^1\times (-1,1)$ be any  $C^1$ curve such that $\tilde \omega(t_1) =  \omega(4\pi)$ and $\tilde \omega(t_2) =  \omega(0)$ such that $\tilde\omega\big((t_1,t_2)\big)$ and ${\omega}\big((0,4\pi)\big)$ are disjoint. Assume that $\tilde \omega$ is parametrized by the arc length. Let us denote
\[
\tilde \omega(t) =\big(\cos\tilde \theta(t),\,  \sin\tilde \theta(t), \, \tilde z(t) \big).
\]

Then,  for each $\theta_0 \in (0,2\pi)$ there is at least one  $t\in(t_1, t_2)$ such that
\[   \tilde \theta(t) = \theta_0 \  (\mbox{mod }2\pi), \quad
 -b \le z(\theta_0)\le \tilde z(t)\le  z(\theta_0+2\pi)\le b , \quad \mbox{and} \quad \tilde \theta'(t) \le 0. \]

As a consequence, using that $\tilde\omega$ is parametrized by the arc length and defining   $A:= \{ t\in (t_1,t_2)\ :\  |\tilde z(t)| \le b,   \   \tilde \theta'(t)\le0\}$ we have
\[
H^1(A) \ge 2\pi.
\]
\end{lem}
\begin{proof}
Note that $|z(4\pi)| \le \frac b 2 + 4\pi \frac{b}{8\pi} \le b$. Let us call $ P=  \omega(0) =  (1, 0, z(0))$ and $Q = \omega(4\pi)=  (1, 0, z(4\pi))$.

For each $\theta_0 \in (0,2\pi)$ the open set
\[\big(S^1\times (-1,1)\big) \setminus \big(  \omega([0,4\pi]) \cup \{(\cos\theta_0, \sin\theta_0)\} \times[-b,b] \big)\]
 has exactly two connected components. The curve $\tilde \omega$, which connects $Q$ and $P$ without intersecting $\omega\big((0,4\pi)\big)$  starts in the upper connected component (the one containing a neighborhood of $Q$) and finishes in the lower connected component (the one containing a neighborhood of $P$). Hence there is at least one time $t_{\theta_0}\in (t_1, t_2)$ at which $\tilde \omega$ intersects the segment $\{(\cos\theta_0, \sin \theta_0)\} \times[-b,b]$ to go from the upper to the lower components. It easily follows that $t_{\theta_0}$ in $A$. 

For the last inequality in the statement, we use that, as shown above,  
the image of $\tilde \omega(A)$  under the the projection of $S^1\times (-1,1) \rightarrow S^1$ has length $2\pi$.
It follows that the length of $\tilde \omega(A)$ is at least $2\pi$, and thus also the length of $A$ (since 
$\tilde \omega$  is parametrized by the arc length).
\end{proof}


We can now give the 

\begin{proof}[Proof of Lemma \ref{lemcurv}]
 Let us assume that $5\pi N\leq L < 5\pi (N+1)$, where $N>0$ is an integer. We need to bound $N$. Hence, we may clearly assume that $N$ is large enough.

 Let us consider the $N$ disjoint intervals $I_j:= [5\pi (j-1), 5\pi j)$, $1\le j\le N$ which are subsets of $[0,L)$.
Let
\[ \kappa_j:=  (1-s)\int_{I_j} \int_0^L |\nu(t)-\nu(\bar t)|^2 k_s\big(\gamma_0(t), \gamma_0(\bar t)\big)\, dt\, d\bar t  \]
and let $j_1,j_2,\dots, j_N$ be an ordering  for which
\[
\kappa_{j_1}\le \kappa_{j_2}\le \cdots\le \kappa_{j_N}.
\]

Choose $M := \lfloor N/2\rfloor$ and notice that
\begin{equation}\label{jjj}
\begin{split}
\max_{1\le i\le M} \kappa_{j_i} &= \kappa_{j_M} \le \frac{1}{N-M} \sum_{i=M+1}^N  \kappa_{j_i}
\\
&=  \frac{1}{N-M}  (1-s) \sum_{i=M+1}^N \int_{I_{j_i}} \int_0^L |\nu(t)-\nu(\bar t)|^2 k_s\big(\gamma_0(t), \gamma_0(\bar t)\big)\, dt\, d\bar t
\\
&\le
\frac{1}{N-M}
(1-s)\int_0^L \int_0^L |\nu(t)-\nu(\bar t)|^2 k_s\big(\gamma_0(t), \gamma_0(\bar t)\big)\, dt\, d\bar t
 \\
&\le 2\frac{C_0(1-s)L}{N},
\end{split}
\end{equation}
where, in the last inequality, we have used assumption \eqref{hp-lemcurv}.

For the sake of clarity, we split the proof in 4 Steps.

 {\em Step 1.} Let us prove that for $I=I_{j_i}$, where  $1\le i\le M= \lfloor N/2\rfloor$, we have
\begin{equation}\label{goals1}
\|\nu(t)-e\|_{C^{1/4}(I)} \leq C \big[\nu\big]_{H^{\sigma}(I)} \le C\delta^{1/2} \quad \mbox{for some }e\in S^2,
\end{equation}
 $\sigma=(1+s)/2$ and 
\[
\delta :=  2\frac{C_0^2}{N}. 
\]

Indeed,  using in \eqref{jjj} the assumption $(1-s)L\le C_0$ we have that,  for $1\le i\le M= \lfloor N/2\rfloor$, the interval $I= I_{j_i}$ has length $5\pi$ and satisfies
\begin{equation}\label{smallerthandelta}
(1-s)\int_I \int_0^L |\nu(t)-\nu(\bar t)|^2 k_s\big(\gamma_0(t), \gamma_0(\bar t)\big)\,dt\, d\bar t\le  2\frac{C_0^2}{N} =\delta.
\end{equation}

Now using \eqref{ineqarcl} we deduce that
 \[
(1-s)\int_I \int_I \frac{|\nu(t)-\nu(\bar t)|^2}{|t-\bar t|^{2+s}}\,dt\, d\bar t\le C \delta,
\]
where $C$ is universal. 

That is, for $\sigma = \frac{1+s}{2}$ and $j=1,2,3$, we have (recall the definition of the $H^\sigma$-seminorm in Lemma \ref{poincare})
\[
\big[\nu^j\big]_{H^{\sigma}(I)} \le C\delta^{1/2}, \quad \mbox{where } \nu=(\nu^1,\nu^2,\nu^3).
\]
Using  Lemma \ref{poincare} we obtain
\[
\|\nu(t)-e\|_{C^{1/4}(I)} \leq C\big[\nu\big]_{H^{\sigma}(I)} \le C\delta^{1/2} \quad \mbox{for some }e\in \R^3.
\]
Note that (for $\delta$ small enough and up to changing $C$) we may assume that $e\in S^2$ since $\nu\in S^2$.
This proves that the velocity (or tangent) vector $\gamma_0'$ is almost perpendicular to $e$ in all of $I$ with a very small error in the angle controlled by $C\delta^{1/2}$. Recall that we may assume $\delta$ to be sufficiently small since, as mentioned in the beginning of the proof, $N$ may be assumed to be large enough.

 {\em Step 2}. We have proven in Step 1 that the restriction of $\gamma_0$ to $I$ is $C\delta^{1/2}$ close to tracing a maximal circle.
As pointed out in the beginning of the proof, we may assume that $N$ is large enough and thus that $\delta = 2C_0^2/N$ is small enough.

Note that since $I$ has length $5\pi>4\pi$, for $\delta$ small enough the curve $\gamma_0\big|_I$ makes two loops at the (topological) cylinder $S^{2}\cap \{ -1/4 < e\cdot x < 1/4\}$, and these loops are $C\delta^{1/2}$  close to the ``equator'' $S^{2}\cap\{e\cdot x =0\}$. In particular the ``vertical'' displacement is less than $C\delta^{1/2}$.
Intuitively,  since $\gamma_0$ is a closed curve it will have to come back again to the starting point of two loops, and since it does not have self-intersections, the only way this may happen is with $\gamma_0$ passing again between the two loops with the opposite orientation (i.e., ``undoing'' the loop).

More precisely, let us prove that
\begin{equation}\label{aaa}
\bar A:= \left\{ \bar t\in [0,L] \setminus I_\circ\ :\  |e \cdot \gamma_0(\bar t)| \le C\delta^{1/2} \  \mbox{and}\   \  e\cdot  \nu(\bar t)\le\frac {1}{100}\right\},
\end{equation}
where $I_\circ\subset I$ is an interval to be defined next and with $|I_0|\ge 3\pi$, satisfies 
\begin{equation}\label{bbb}
H^1(\bar A)\ge \frac{19}{10} \pi.
\end{equation}

Indeed, let us choose an orthonormal coordinate frame $X,Y, Z$ with origin at $0$ and with $Z$ directed as $e$.
Let us define  ``cylindrical'' coordinates  in $S^{2}\cap \{ -1/4 < e\cdot x < 1/4\}$ as follows
\[ X= \cos\theta \cos z, \ Y= \sin \theta \cos z, \ Z=\sin z.\]
Since $\gamma_0$ is a closed curve without self-intersections we may apply Lemma \ref{topolog} with
\[
\omega = \gamma_0|_{I_\circ}\,, \qquad  \tilde \omega = \gamma_0|_{ ([0,L]/\{0,L\}) \setminus {I_\circ}}\,,\qquad \mbox{and}\qquad b= C\delta^{1/2},
\]
where ${I_\circ}\subset I$ is an interval for which $\int_{I_\circ} \theta'( \gamma_0(t)) dt = 4\pi$ as in Lemma \ref{topolog} ---here we abuse notation and omit the fact that $\omega$ and $\tilde \omega$ would need to  reparametrized by the angle $\theta$ and by the arc length of the cylinder respectively.  
From the last equality we deduce, using $1= |\gamma_0'| = (\theta')^2\cos^2 z+ (z')^2 \ge (\theta')^2 (3/4)^2$ 
if $\delta$ is small enough, that $|I_\circ|\ge 3\pi$.  

Applying Lemma \ref{topolog}, the set
\[ A:= \{ \bar t\in ([0,L]/\{0,L\})\setminus I_\circ\ :\  |\tilde z(\bar t)| \le C\delta^{1/2},   \   \tilde \theta'(\bar t)\le0\}\] ---in the notation of Lemma \ref{topolog}---
satisfies $H^1(A)  \ge  \frac{19}{10} \pi$.
Here, on the right hand side  we need to choose number slightly smaller than $2\pi$ due to the fact 
that $\tilde\omega$ needs to be reparametrized by the arc length of the cylinder in oder to apply 
Lemma~\ref{topolog} (understanding that $\delta$ is chosen accordingly 
small enough so that the arc lengths on the sphere near the equator and on cylinder are almost the same).

Observe also that for every $ \bar t\in A$ we have that $|\tilde z(\bar t)|$ is very small (for $\delta$ small enough) and that $\tilde \theta'(\bar t)\leq 0$. As a consequence\footnote{{Note that if it was $|z(\bar t)|=0$ as some $\bar t$ the condition $\tilde\theta'(\bar t)\le0$  would be exactly equivalent to $\nu(\bar t)\cdot e\le 0$. Therefore, if $|z(\bar t)|$ is very small $ e\cdot \nu(\bar t)$ cannot bee too positive.}}
$$
 e\cdot \nu(\bar t)\le  \frac{1}{100},
$$
as before in \eqref{aaa}, provided that $\delta$ is small enough.
In other words, the normal vector to $\gamma_0$  at $\bar t$, which is tangent to $S^{2}$, can only have, at most, a tiny positive projection in the ``vertical'' direction  $e$. 

Hence, $A\subset \bar A$ and  \eqref{bbb} follows.

{\em Step 3.}
For each given $\bar t\in \bar A$ there exists $t'\in I_\circ$ such that  $|\gamma_0(\bar t)- \gamma_0(t')| \le C\delta^{1/2}$, with $C$ universal, since $\gamma_0|_{I_\circ}$ makes two full loops to the equator.
Hence, we deduce that
\begin{equation}\label{aaaa}
\begin{split}
\int_{I_0}  \frac{1}{ \big|\gamma_0(t)- \gamma_0(\bar t)\big|^{2+s}} \, dt &\ge c \int_{I_0}  \frac{1}{ \big(\delta^{1/2} +|\gamma_0(t)- \gamma_0(t')| \big)^{2+s}} \, dt
\\
&\geq c\int_{I_0}  \frac{1}{ \big(\delta^{1/2} +|t-t'| \big)^{2+s}} \, dt\\
&\ge
c  \int_{0}^{\frac{3\pi}{2}}  \frac{1}{ \big(\delta^{1/2} + \tau \big)^{2+s}} \, d\tau \ge c (\delta^{1/2}  )^{-1-s},
\end{split}
\end{equation}
where in the third inequality we have used that $t'\in I_0$ and that $I_0$ is an interval of length at least $3\pi$.

Now, notice that for all $t\in I_0$ and $\bar t\in  \bar A$ we have $\big|\nu(t)-\nu(\bar t)\big| \ge 1$ 
---since the angle between $\nu(t)$ and $\nu(\bar t)$ is at least of $85^o$. In addition, 
recall \eqref{smallerthandelta} and \eqref{bbb} to~obtain
\begin{equation}\label{bbbb}
\begin{split}
\delta &\ge (1-s)\int_{I_0} \int_{\bar A} \big|\nu(t)-\nu(\bar t)\big|^2 k_s\big(\gamma_0(t), \gamma_0(\bar t)\big)\, d\bar t \, dt
\\
&\ge (1-s) c\int_{\bar A} \int_{I_0} \frac{1}{ \big|\gamma_0(t)- \gamma_0(\bar t)\big|^{2+s}} \, dt\, d\bar t
\\
&\ge  (1-s)  \frac{cH^1(\bar A)}{(\delta^{1/2})^{1+s}}
\\
& \ge  (1-s)\frac{c}{(\delta^{1/2})^{1+s}}
\end{split}
\end{equation}
for different universal constants $c>0$.

It follows that
\[ \left(2\frac{C_0^2}{N}\right)^{-1-\frac{1+s}{2}} = \delta^{-1-\frac{1+s}{2}} \le \frac{C}{1-s}.\]
Hence, using that $s\ge 1/2$, we obtain
\[ L\le 5\pi (N+1) \le  \frac{C}{(1-s)^{4/7}},\]
where $C$ depends only on $C_0$.

{\em Step 4.} Next we repeat exactly the same argument as in Steps 1, 2, and 3   but now using \eqref{jjj} together with the improved estimate $L \le C(1-s)^{-4/7}$ instead of $L\le C_0(1-s)^{-1}$.
We now have that, for $1\le i\le M= \lfloor N/2\rfloor$, the interval $I= I_{j_i}$ has length $5\pi$ and satisfies
\begin{equation}\label{smallerthandelta1}
\begin{split}
(1-s)\int_I \int_0^L |\nu(t)-\nu(\bar t)|^2 k_s\big(\gamma_0(t), \gamma_0(\bar t)\big)\,dt\, d\bar t &\le 2\frac{C_0(1-s)L}{N} \\
&\le \frac{C(1-s)^{3/7}}{N}=:\delta'.
\end{split}
\end{equation}

Therefore, arguing exactly  as above, we obtain
\[ \left(\frac{C(1-s)^{3/7}}{N}\right)^{-1-\frac{1+s}{2}} = (\delta')^{-1-\frac{1+s}{2}} \le \frac{C}{1-s},\]
where $C$ depends only on $C_0$.

Hence
\[ \frac{N}{(1-s)^{3/7}} \le  \frac{C}{(1-s)^{4/7}}\]
and thus
\[ L\le 5\pi (N+1) \le  \frac{C}{(1-s)^{1/7}},\]
where $C$ depends only on $C_0$.

Finally we repeat exactly the same argument once more but now using \eqref{jjj} together with the improved estimate $L \le C(1-s)^{-1/7}$ instead of $L \le C(1-s)^{-4/7}$.
We now have that, for $1\le i\le M= \lfloor N/2\rfloor$ the interval $I= I_{j_i}$ satisfies
\begin{equation}\label{smallerthandelta2}
\begin{split}
(1-s)\int_I \int_0^L |\nu(t)-\nu(\bar t)|^2 k_s\big(\gamma_0(t), \gamma_0(\bar t)\big)\,dt\, d\bar t& \le \ 2\frac{C_0(1-s)L}{N} \\
&\le \frac{C(1-s)^{6/7}}{N}=:\delta''.
\end{split}
\end{equation}

Therefore, 
\[ \left(\frac{C(1-s)^{6/7}}{N}\right)^{-1-\frac{1+s}{2}} = (\delta'')^{-1-\frac{1+s}{2}} \le \frac{C}{1-s},\]
and we conclude 
\[ \frac{N}{(1-s)^{6/7}} \le  \frac{C}{(1-s)^{4/7}}\]
and 
\[ L\le 5\pi (N+1) \le  C(1-s)^{2/7},\]
where $C$ depends only on $C_0$.

Note that when $s\uparrow 1$ the previous inequality does not really lead to a contradiction since, to obtain it, we assumed that $L \ge 5\pi N$ with $N\ge 1$ large enough (depending on $C_0$).
It follows from this observation that $L\le C$ for $s$ close to $1$, where $C$ depends only on $C_0$.
 \end{proof}

Finally, once we know that $L\le C$, with $C$ are universal and in particular independent of $s$ for $s\in(1/2,1)$, we  conclude from the integral control on the squared nonlocal second fundamental form that $\gamma_0$  converges in $C^{1,1/4}$ norm to a maximal circle as $s\uparrow 1$. This is the content of the next result.
\begin{lem}\label{lemcurv2}
Let $s\in(1/2,1)$,  $L>0$, $\gamma_0$, and $\nu$ be as in Lemma \ref{lemcurv}. In particular, we assume that, for some constant $C_0$,
\begin{equation}\label{hp-lemcurv2}
\int_0^L \int_0^L |\nu(t)-\nu(\bar t)|^2 k_s\big(\gamma_0(t), \gamma_0(\bar t)\big)\, dt\, d\bar t\le C_0 L.
\end{equation}

Assume in addition that
\[ 0<   L\le C_0.\]

Then,  $L\rightarrow 2\pi$ as $s\uparrow 1$ and, for some $e\in S^2$, we have
\[  \|\nu-e\|_{C^{1/4}([0,L]/\{0,L\})} \le C(1-s)^{1/2}\]
where $C$ is some constant depending only on $C_0$.

In, particular, for $s$ close enough to $1$, the cone generated by (the image of) $\gamma_0$ is a very flat Lipschitz graph.
\end{lem}

\begin{proof}
Since $\gamma_0$ is a closed curve, let us reparametrize it as follows:
\[
\hat \gamma_0 : S^1 \rightarrow S^2  \quad\mbox{where }\hat\gamma_0(\theta) =  \gamma_0\left(\frac{L}{2\pi}\theta \right), \  \theta \in S^1 \cong \R/\{2\pi \mathbb Z\},
\]
where we are identifying  $\R/\{2\pi \mathbb Z\}$ and  $\theta\in S^1$ via the isometry $\theta\mapsto (\cos\theta, \sin\theta)$.

Similarly as in the proof of Lemma \ref{lemcurv}, defining $\sigma = \frac{1+s}{2}$ and using now that $L\leq C_0$, we obtain
\[
\begin{split}
\big[\hat \nu^i\big]^2_{H^{\sigma}(S^1)}& = (1-\sigma) \int_{S^1}\int_{S^1}  \frac{ |\hat\nu(\theta)-\hat\nu(\bar \theta)|^2}{|(\cos\theta, \sin\theta)-(\cos\bar\theta,\sin\bar\theta)|^{1+2\sigma}}  \, d\theta \, d\bar \theta
 \\
 &\le C(1-s) L^{s}\int_0^L \int_0^L |\nu(t)-\nu(\bar t)|^2 k_s\big(\gamma_0(t), \gamma_0(\bar t)\big)\, dt\, d\bar t\le C(1-s)
\end{split}
\]
where 
$\hat \nu(\theta) = \nu\left(\frac{L}{2\pi}\theta\right)$ is the normal vector accordingly reparametrized, and where $C$ depends only on $C_0$. Here we have used that $t = \frac{L}{2\pi} \theta$, $\bar t = \frac{L}{2\pi} \bar\theta$,  $\hat \gamma_0'(\theta) = \frac{L}{2\pi}$, 
\[
\frac{L}{2\pi} |(\cos\theta, \sin\theta)-(\cos\bar\theta,\sin\bar\theta)| \ge |\hat \gamma_0(\theta)-\hat \gamma_0(\bar\theta)| = |\gamma_0(t)- \gamma_0(\bar t) |,
\]
and \eqref{ineqarcl}.

Using a small variation of Lemma \ref{poincare} ---for $S^1$ instead of an interval--- we obtain
\begin{equation}\label{last}
\|{\hat\nu}(t)-e\|_{C^{1/4}(S^1)} \leq C \big[{\hat\nu}\big]_{H^{\sigma}(S^1)} \le C(1-s)^{1/2} 
\end{equation}
for some $e\in S^2$. 

It follows that  $\hat\nu(t)$ is almost parallel to $e$ and thus $\gamma_0$ is a small perturbation of a maximal circle. In particular, $L\to 2\pi$ as $s\uparrow1$. Note also that the cone generated by $\gamma_0$ is a Lipschitz graph in the direction $e$ for $s$ sufficiently close to 1.
\end{proof}

\section{Proof of main theorem}
In this section we finally give the
\begin{proof}[Proof of Theorem \ref{thmcones3}]
Recall that by assumption $\Sigma$ is a stable minimal cone for the $s$-perimeter in $\R^3$.
Let us call $\gamma := \partial{\Sigma}\cap S^2$, where $S^2:= \{x\in \R^3\, :\, |x|=1\}$.

The curve $\gamma$ can be written as a disjoint union $\gamma = \gamma_1 \cup \gamma_2\cup \cdots\cup \gamma_J$, where $\gamma_i$ are closed $C^2$ oriented curves, each of them connected and without self-intersections.

Let $L_i$ denote the length of $\gamma_i$ and $L= \sum_{1\le i\le J} {L_i}$.
Applying Theorem \ref{perestimate} to the cone $\partial \Sigma$, we deduce that
\begin{equation}\label{roughperest}
L \le \frac{C}{1-s}.
\end{equation}
Throughout the proof $C$ will denote, possibly different, positive constants depending only on $s$ and bounded as $s \uparrow 1$.

{\em Step 1.}  Let us consider first the case $J=1$. In this case $\gamma$ is a connected closed curve.
By Proposition \ref{crucial3d} we have
\begin{equation}\label{key1A}
\iint_{\gamma\times \gamma} |\nu(\hat x)-\nu(\hat y)|^2 k_s(\hat x, \hat y) dH^1(\hat x) dH^1(\hat y)  =  \int_{\gamma} dH^1(\hat x) c^2_{\partial \Sigma}(\hat x)  \le CL,
\end{equation}
where $L$ is the length of $\gamma$. By Lemma \ref{lemcurv}, we know that $L\leq C$.

Therefore, using Lemma \ref{lemcurv2} we prove that, if $s\in (0,1)$ is close enough to $1$, then $\gamma$ is a small $C^{1,1/4}$ deformation of a maximal circle and thus $\partial \Sigma$ is a Lipschitz graph.

Since $\partial \Sigma$ is $C^2$ and stable away from $0$ then it is a viscosity solution of the fractional minimal surface equation in $\partial \Sigma\setminus \{0\}$.
Then,  since $\partial\Sigma$ is a cone  it must be\footnote{If a $C^2$ surface touches a cone at $0$ 
then the cone is contained in a half-space. Thus,  the convex envelope of the cone is a subsolution 
(of the fractional minimal surface equation) that touches $\partial \Sigma$ 
by below along generatrices. Since, $\partial\Sigma$ is $s$-minimal away from $0$  
the strong maximum principle yields that $\partial\Sigma$ must be a plane in such situation} a viscosity solution also at $0$.

Then, by the standard foliation argument (using the uniqueness of viscosity solutions of the fractional minimal surface equations among graphs), $\Sigma$ is a minimizer of the $s$-perimeter (and not just a stable set). Since now we know that $\partial \Sigma$ is a Lipschitz $s$-minimal graph, we can apply Theorem 1.1 in \cite{FV} and deduce that $\Sigma$ is $C^{\infty}$ and hence, being a cone, it is necessarily a hyperplane.

{\em Step 2.} Let us now assume that $J>1$ and reach a contradiction. Now,  \eqref{key1A} reads
\begin{equation}\label{key11}
\sum_{1\le i\le J} \int_{\gamma_i} dH^1(\hat x) \int_{\gamma} dH^{1}(\hat y)\, \big|\nu_\Sigma(\hat x)-\nu_\Sigma(\hat y)\big|^2 \,  k_s(\hat x, \hat y)   \le C\sum_{1\le i \le J} L_i.
\end{equation}

For each $i$ let
\[q_i : =  \frac{1}{L_i} \int_{\gamma_i} dH^1(\hat x) \int_{\gamma}  dH^{1}(\hat y) \,\big|\nu_\Sigma(\hat x)-\nu_\Sigma(\hat y)\big|^2 \,  k_s(\hat x, \hat y).\]
Without loss of generality let us assume that  $q_{1}\le q_{2}\le\dots\le q_{J}$, after relabeling the indexes.

By \eqref{key11} we have
\[
\frac{\sum_{1\le i\le J} L_i q_i }{\sum_{1\le i\le J} L_i  }\le  C
\]
and hence,
\[ q_1 \le C.\]
Then, Lemmas \ref{lemcurv} and \ref{lemcurv2} yield
\[
0< \pi\le L_1 \leq C,
\]
with $C$ universal,  for $s$ sufficiently close to $1$.

It follows, by \eqref{key11}, that
\begin{equation}\label{blabla}
\sum_{i=2}^J L_i q_i \ge  C\left( C +\sum_{i=2}^J  L_i \right).
\end{equation}
Note that we have $\sum_{2\le i\le J} L_i\geq \pi$. Indeed, if this were not true, we would have
\[
\int_{\gamma_2} dH^1(\hat x) \int_{\gamma_2} dH^{1}(y) \big|\nu_\Sigma(\hat x)-\nu_\Sigma({\hat y})\big|^2 \,  k_s(\hat x, \hat y)    \le C
\]
and $L_2<\pi$. The proof of Lemma \ref{lemcurv2} then gives that $L_2$ is close to $2\pi$ if $s$ is sufficiently close to $1$ ---a contradiction with $L_2<\pi$.

Therefore, \eqref{blabla} yields
\[
\sum_{i=2}^J L_i q_i \le  C \sum_{i=2}^J L_i
\]
and thus
\[ q_2 \le C.\]
Then, using again Lemma  \ref{lemcurv}  we find that
\[
 L_2 \le C,
\]
with $C$ universal.

Next, using Lemma  \ref{lemcurv2}, we have
\[
\|\nu_i (t)-e_i\|_{C^{1/4}(I_i)} \le C(1-s)^{1/2}, \quad \mbox{for some }e_i \mbox{ in }S^2
\]
and $i=1,2$, where $\nu_i(t)$ is the normal to $\gamma_i$ at $\gamma_i(t)$ ---recall that $\gamma_i$ are parametrized by the arc length in an interval $I_i$.
Since the two curves do not intersect and are $C(1-s)^{1/2}$ close to maximal circles, we must have either 
\[
|e_1-e_2| \le C(1-s)^{1/2} \qquad \mbox{or} \qquad |e_1+e_2| \le C(1-s)^{1/2}.
\]
In other words, the two curves are very close to the the same maximal circle (in $C^{1,1/4}$ norm), but they may have either the same or opposite orientation.

In the second case (opposite orientations) we use that $q_1 L_1\le   C$ and reason exactly as in Step 3 of the proof of Lemma \ref{lemcurv} ---more precisely, as in \eqref{aaaa} and \eqref{bbbb}--- to obtain
\begin{equation}\label{final}
\begin{split}
\frac{1}{C \big( (1-s)^{1/2}\big)^{1+s} } &\le   \int_{\gamma_1} dH^1(\hat x) \int_{\gamma_2}dH^{1}(\hat y) \frac{ 1}{|\hat x-\hat y|^{2+s}} 
\\
&\le   \int_{\gamma_1} dH^1(\hat x) \int_{\gamma_2}dH^{1}(\hat y) \frac{ \big|\nu_\Sigma(\hat x)-\nu_\Sigma(\hat y)\big|^2 }{|\hat x-\hat y|^{2+s}} 
\\
&\le  C\int_{\gamma_1} dH^1(\hat x) \int_{\gamma_2} dH^{1}(\hat y)\big|\nu_\Sigma(\hat x)-\nu_\Sigma(\hat y)\big|^2 \, k_s(\hat x,\hat y) 
\\
& \le q_1 L_1 \le C.
\end{split}
\end{equation}
This yields to a contradiction if $s$ is close to~$1$.

In the first case, if the two curves $\gamma_1$ and $\gamma_2$ happen to have the same orientation,  since $\gamma \subset S^2$ is a boundary (of the set $\Sigma\cap S^2$) then there must be a third curve $\gamma_{j_*}$ with the opposite orientation and trapped between $\gamma_1$ and $\gamma_2$. In this case, reasoning as in  \eqref{final} with $\gamma_2$ replaced by  $\gamma_{j_*}$ we reach a contradiction if $s$ is close to~$1$.
Note tough that here we need to be a bit more careful since we have not proven that $\gamma_{j_*}$ is very close to the maximal circle in $C^{1,1/4}$ norm but just in Hausdorff distance (we know that it is trapped between two small perturbations of a maximal circle). However, we can proceed exactly as we did in \eqref{aaa}: define the set $\bar A$ of times $\bar t$
such that $e_1\cdot \nu_{j_*}(\bar t) \le \frac{1}{100}$, which will satisfy \eqref{bbb},  and repeat \eqref{final} but  
integrating only on the set $\{\hat y \in \gamma_{j_*}(\bar A)\}$ and not along the whole $\gamma_{j_*}$. 
Doing so we guarantee that $|\nu_\Sigma(\hat x)-\nu_\Sigma(\hat y)| \ge 1$ and the computation  
would be again identical as in \eqref{aaaa} and \eqref{bbbb}.

\end{proof}

\end{document}